\documentclass[reqno, 12pt]{amsart}
\usepackage{amsmath}
\usepackage{amssymb}
\usepackage{amsthm}
\usepackage{mathtools}
\usepackage{tikz-cd}
\usepackage{enumerate}
\usepackage[mathscr]{eucal}
\usepackage{graphicx}

\newtheorem{theorem}{Theorem}[section]
\newtheorem{lemma}[theorem]{Lemma}
\newtheorem{proposition}[theorem]{Proposition}
\newtheorem{corollary}[theorem]{Corollary}
\newtheorem{definition}[theorem]{Definition}
\newtheorem{definitions}[theorem]{Definitions}
\theoremstyle{definition}
\newtheorem{example}[theorem]{Example}
\newtheorem{remark}[theorem]{Remark}

\newcommand{\trinorm}{\vert\hspace{-0.5mm}\vert\hspace{-0.5mm}\vert}

\begin{document}
	\title[]{Dual of an extended locally convex space}
	\author{Akshay Kumar \and Varun Jindal}
	\address{Akshay Kumar: Department of Mathematics, Malaviya National Institute of Technology Jaipur, Jaipur-302017, Rajasthan, India}
	\email{akshayjkm01@gmail.com}
	
	\address{Varun Jindal: Department of Mathematics, Malaviya National Institute of Technology Jaipur, Jaipur-302017, Rajasthan, India}
	\email{vjindal.maths@mnit.ac.in}

\maketitle

\begin{abstract}
	This paper aims to study the dual of an extended locally convex space. In particular, we study the weak and weak$^*$ topologies as well as the topology of uniform convergence on bounded subsets of an extended locally convex space. As an application to function spaces, we show that the weak topology for the space $C(X)$ of all real-valued continuous functions on a metric space $(X,d)$ endowed with the topology of strong uniform convergence on bornology coincides with its finest locally convex topology if and only if the bornology is of all finite subsets of $X$.    \end{abstract}

\section{Introduction}
In \cite{nwiv}, Beer introduced the concept of an extended normed linear space (enls, for short). These spaces were further studied by Beer and Vanderwerff in \cite{socsiens, spoens}. Extending the idea of Beer \cite{nwiv}, Salas and Tapia-Garc{\'\i}a \cite{esaetvs} defined extended locally convex space (elcs, for short). These new extended spaces are different from the classical locally convex spaces in that the scalar multiplication in these spaces need not be jointly continuous. So the classical theory of locally convex spaces may not directly apply to these new kinds of spaces. To address this issue,  in  \cite{flctopology}, authors studied the finest locally convex topology (flc topology, for short) for an elcs $(X, \tau)$ which is coarser than $\tau$.  There it is also shown that if $\tau_F$ is the flc topology for an elcs $(X, \tau)$, then both $(X, \tau)$ and $(X, \tau_F)$ have the same dual $X^*$ (the collection of all continuous linear functionals).

In the present paper, we explore the use of the flc topology $\tau_F$ for an elcs $(X, \tau)$ to study its dual $X^*$. In particular, we study the weak topology on $X$, and the weak$^*$ topology on the dual $X^*$ of $X$. On $X^*$, we also study the topology $\tau_{ucb}$ of uniform convergence on bounded subsets of an elcs $(X, \tau)$. The topology $\tau_{ucb}$ for extended normed spaces was first considered by Beer and Vanderwerff in \cite{spoens}.

After giving some preliminary definitions and notations in the second section, we study the weak and weak$^*$ topologies for an elcs in the third section. In particular, we study the metrizability and normability of these topologies. To investigate the metrizability of the weak topology for an elcs $(X,\tau)$, we first examine when the flc topology coincides with its weak topology. Further, we show that the weak$^*$ topology on the  closed unit ball in the dual of an extended normed space $(X,\parallel\cdot\parallel)$ is metrizable if and only if $(X, \tau_F)$ is separable, where $\tau_F$ is the corresponding flc topology. As an application to function spaces, we show that if $\mathcal{B}$ is a bornology on a metric space $(X,d)$ with a closed base, then the flc topology and the weak topology for the space $C(X)$ of all real-valued continuous functions equipped with the topology $\tau_{\mathcal{B}}^s$ of strong uniform convergence on $\mathcal{B}$ coincide if and only if $\mathcal{B}$ is the bornology of all finite subsets of $X$.

In the fourth section, we examine equicontinuous and weak$^*$ compact subsets of the dual $X^*$ of an elcs  $(X, \tau)$, and show that the collection of all equicontinuous subsets of $X^*$ actually induces $\tau$. We also verify the classical Banach Alaoglu Theorem, James Compactness Criterion, Eberlin-\v{S}mulian Theorem in the extended case.  

In the last section, we study the topology $\tau_{ucb}$ of uniform convergence on bounded sets of an elcs $(X, \tau)$. We first examine the metrizability and normability of the space $(X^*, \tau_{ucb})$. Then we investigate the following interesting problem. For an elcs $(X, \tau)$ with the flc topology $\tau_F$, the spaces $(X,\tau)$ and $(X, \tau_F)$ may have different bounded sets (see, Proposition \ref{both tau and tauF may not have same bounded sets}), it will be interesting to know that when  $\tau_{ucb}= \tau_s$ on $X^*$, where $\tau_s$ is the strong topology of the locally convex space $(X,\tau_F)$, that is, the topology of uniform convergence on bounded subsets of $(X, \tau_F)$. We show that this is true for a fundamental extended locally convex space. We also give conditions on an elcs so that $\tau_{ucb}=\tau_s$. 

\section{Preliminaries}

Every vector space is assumed to be over the field $\mathbb{K}$, which is either $\mathbb{R}$ or $\mathbb{C}$. We adopt the following conventions for the $\infty$ value: $\infty.0=0.\infty=0$; $\infty+\alpha=\alpha+\infty=\infty$ for every $\alpha\in\mathbb{R}$; $\infty.\alpha=\alpha.\infty=\infty$ for $\alpha>0$; $\inf\{\emptyset\}=\infty$.  

An \textit{extended norm} $\parallel\cdot\parallel:X\rightarrow [0,\infty]$ on a vector space $X$ is a function  with the following properties: \begin{itemize}
	\item[(1)] $\parallel x\parallel=0$ if and only if $x=0$;
	\item[(2)] $\parallel \alpha x\parallel = |\alpha |\parallel x\parallel$ for each scalar  $\alpha$ and $x\in X$;
	\item[(3)] $\parallel x+y\parallel~\leq~ \parallel x\parallel+\parallel y\parallel$ for $x,y\in X$.\end{itemize}

An \textit{extended normed linear space} (enls, for short) is a vector space $X$ endowed with an extended norm $\parallel \cdot\parallel$, denoted by $(X,\parallel \cdot\parallel)$ (or $X$). An extended Banach space is an enls $(X,\parallel\cdot\parallel)$ that is complete with respect to the metric $d(x, y)=\min\{1, \parallel x-y\parallel\}$ for $x, y\in X$. For more details on enls and extended metric spaces, we refer to \cite{tsoervms, nwiv, socsiens, spoens}.
 
A generalization of an extended norm is an extended seminorm. 
  
\begin{definition}{\normalfont(\cite{esaetvs})} \normalfont	A function $\rho:X\rightarrow [0, \infty]$ on a vector space $X$ is said to be an \textit{extended seminorm} if it satisfies: \begin{itemize}
		\item[(1)] $\rho(\alpha x)=|\alpha|\rho(x)$ for each scalar $\alpha$ and $x\in X$; 
		\item[(2)] $\rho(x+y)\leq \rho(x)+\rho(y)$ for $x,y\in X$.\end{itemize}\end{definition}
	 
\noindent A vector space $X$ equipped with a group topology $\tau$ is said to be an \textit{extended locally convex space} (elcs, for short) if $\tau$ is induced by a family $\mathcal{P}=\{\rho_i : i\in I \}$ of extended seminorms on $X$. If the family $\mathcal{P}$ is countable, then we say $(X,\tau)$ is a \textit{countable elcs}. We denote by $S(X, \tau)$ the collection of all continuous extended seminorm on an elcs $(X, \tau)$. For $\rho\in S(X, \tau)$, we define $X_{fin}^\rho=\{x\in X:\rho(x)<\infty\}$. The \textit{finite subspace} of an elcs $(X, \tau)$ is given by $$X_{fin}=\{x\in X: \rho(x)<\infty ~\text{for every}~ \rho\in S(X, \tau)\}.$$ In particular, if $(X,\parallel\cdot\parallel)$ is an enls, then $X_{fin} = \{x \in X: \parallel x \parallel < \infty\}$.

If there exists a $\rho\in S(X, \tau)$ such that $X_{fin}=X_{fin}^\rho$ in an elcs $(X, \tau)$, then we say $(X, \tau)$ is a \textit{fundamental extended locally convex space} (fundamental elcs, for short). The following facts about an elcs $(X,\tau)$ are either easy to prove or given in \cite{esaetvs}.
\begin{itemize}
	\item[(1)] $(X_{fin}, \tau|_{X_{fin}})$ is a locally convex space, and it is the largest such subspace of $X$.
	\item[(2)]  If $\rho$ is a continuous extended seminorm on $X$, then $X_{fin}^\rho$ is a clopen in $X$.
    \item[(3)]  $X_{fin}$ is an open subspace of $X$ if and only if $X$ is  a fundamental elcs (Corollary 3.11 in \cite{esaetvs}). 
	\item[(4)]  If $\tau$ is induced by a collection $\mathcal{P}$ of  extended seminorms on $X$, then $$X_{fin}=\bigcap_{\rho\in\mathcal{P}} X_{fin}^\rho.$$
   \item[(5)] There exists a neighborhood base $\mathfrak{B}$ of $0$ consisting of absolutely convex (balanced and convex) sets.
\end{itemize} 

Suppose $\mathfrak{B}$ is a neighborhood base of $0$ in an elcs $(X,\tau)$ consisting of absolutely convex sets. Then $\tau$ is induced by the collection $\{\mu_{V}:V\in \mathfrak{B}\}$ of Minkowski functionals (see, Proposition 4.7 in \cite{esaetvs}).

\begin{definition}\label{Minkowski funcitonal} {\normalfont(\cite{esaetvs})} \normalfont	Suppose $(X, \tau)$ is an elcs and $U$ is an absolutely convex subset of  $X$. Then the \textit{Minkowski functional} $\mu_U:X\rightarrow[0,\infty]$  for $U$  is defined as  $$\mu_U(x)=\inf\{\lambda>0: x\in \lambda U\}.$$\end{definition}

\noindent For an absolutely convex subset $U$ of an elcs $X$, we define  $X_{fin}^U=\{x\in X: \mu_U(x)<\infty\}$. The following facts are easy to prove using Definition \ref{Minkowski funcitonal}. 
\begin{itemize}
	\item[(1)] Every Minkowski functional $\mu_{U}$ is an extended seminorm. In addition, suppose $0\in \text{Core}(U)$, that is, for  every $x\in X$, there exists a $\delta_x>0$ such that $tx\in A$ for every $0\leq t\leq \delta_x.$ Then $\mu_U$ is a seminorm on $X$.
	\item[(2)] The Minkowski functional $\mu_U$ is continuous if and only if   $U$ is a neighborhood of $0$. 
	\item[(3)]  If $U\subseteq V$, then $\mu_V\leq \mu_U$. 
	\item[(4)]  In general, $\{x\in X:\mu_U(x)<1\}\subseteq U\subseteq \{x\in X:\mu_U(x)\leq 1\}.$ \end{itemize}

For any nonempty subset $A$ of a vector space $X$, we denote the absolute convex hull of $A$ in $X$ by ab$(A)$, that is, ab$(A)$ is the smallest absolutely convex set containing $A$. If $A$ is any set in a  topological space $(X,\tau)$, we denote the closure and interior of $A$ in $(X,\tau)$ by $\text{Cl}_\tau(A)$ and  $\text{int}_\tau(A)$, respectively. For other terms and definitions, we refer to \cite{lcsosborne, tvsschaefer, willard}. 

\section{Weak and Weak$^*$ Topologies}

This section aims to study some structural properties of the weak and weak$^*$ topologies for an elcs. In particular, we study the metrizability and normability of these topologies in relation to the properties of an elcs. We also show that for an enls $(X,\parallel\cdot \parallel)$, the weak$^*$ topology on the closed unit ball $B_{X^*}$ of the dual $X^*$ is metrizable if and only if $X$ endowed with the flc topology $\tau_F$ is separable.

\begin{definitions}\label{weak and weak^* topology}
\normalfont	Let $(X,\tau)$ be an elcs with dual $X^*$. The \textit{weak topology} $\tau_w$ on $X$ is induced by the collection $\mathcal{B}_w=\{\rho_f :f\in X^*\}$ of seminorms, where $\rho_f(x)=|f(x)|$ for all $x \in X$. 

We define the \textit{weak$^*$ topology} $\tau_{w^*}$ on $X^*$ as the topology induced by the collection $\mathcal{B}_{w^*}=\{\rho_x :x\in X\}$ of seminorms, where $\rho_x(f)=|f(x)|$ for all $f \in X^*$.\end{definitions}

\begin{remark}\label{finite space have same weak  and weak$^*$ topology}
\begin{enumerate}
	\item If $\tau_{F}$ is the flc topology for an elcs $(X, \tau)$, then $(X,\tau_{F})^*=(X,\tau)^*$.  So the weak topology $\tau_w$, and the weak$^*$ topology $\tau_{w^*}$ are same as the weak and weak$^*$ topologies for the locally convex space $(X, \tau_F)$. Hence we may use most of the classical results about the weak and weak$^*$ topologies of a locally convex space for the topologies $\tau_w$ and $\tau_{w^*}$. 
	\item The collection $\mathcal{B}_w=\{B^\circ: B ~\text{is a finite subset of}~ X^*\}$ is a neighborhood base of $0$ in $(X, \tau_w)$, where $B^\circ=\{f\in X^*: |f(x)|\leq 1 ~\text{for}~ x\in B\}$. Similarly, the collection $\mathcal{B}_{w^*}=\{B_\circ: B ~\text{is a finite subset of}~ X^*\}$ is a neighborhood base of $0$ in $(X, \tau_{w^*})$, where $B_\circ=\{x\in X: |f(x)|\leq 1 ~\text{for}~ f\in B\}$.  
\end{enumerate}

\end{remark}      
For more details about weak and weak$^*$ topologies for a locally convex space, we refer to  \cite{tvsnarici, lcsosborne, tvsschaefer}.

It is well known that  a locally convex space  is normable if and only if it has a bounded neighborhood of $0$ (Theorem 6.2.1 in \cite{tvsnarici}, p. 160), and a neighborhood of $0$ in the weak topology of a normed space $X$ is bounded if and only if the dimension of $X$ is finite (Proposition 3.10 in \cite{faaidg}, p. 67). So the weak topology of a locally convex space is normable if and only if it is finite dimensional. Consequently, by Remark \ref{finite space have same weak  and weak$^*$ topology}, we have the following result for an elcs.
\begin{proposition}\label{normability of the weak topology}
	Let $(X, \tau)$ be an elcs. Then the weak topology for  $(X, \tau)$ is normable if and only if the dimension of $X$ is finite.\end{proposition}  

In the next result,  we examine when $\tau_{F} =\tau_{w}$ for an elcs $(X,\tau)$? 
\begin{proposition}\label{coinicidence of finest and weak topology}
	Let $(X, \tau)$ be an elcs with the flc topology $\tau_{F}$ and weak topology $\tau_w$. If $\tau_{F}=\tau_w$, then the dimension of the quotient space $X/X_{fin}^\rho$ is finite for every continuous extended seminorm $\rho$ on $(X,\tau)$. \end{proposition}
\begin{proof} Let $X=X_{fin}^\rho\oplus M$ for some infinite dimensional subspace $M$ and continuous extended seminorm $\rho$. So there exists an infinite linearly independent subset $A=\{z_n:n\in\mathbb{N}\}$ of $M$. Since $\tau_{F}=\tau_w$, every discrete subgroup of $(X, \tau_{F})$ must be finitely generated (\cite{Atgcfwt}). If $Z$ be the subgroup of $X$ generated by $A$, then $Z$ cannot be finitely generated. To get a contradiction, we show that $Z$ is a discrete subgroup of $(X, \tau_{F})$. Let $X= Y\oplus ~\text{span}(A)$, for a subspace $Y$ of $X$. Clearly, $X_{fin}^\rho\subseteq Y$. Consider the seminorm $\mu:X\rightarrow [0,\infty)$ defined by
$$\mu(x)=\sum_{n\in\mathbb{N}}|\alpha_n|,$$
where $x=x_1+x_2$ for $x_1\in Y$ and  $x_2 = \sum_{n\in\mathbb{N}}\alpha_nz_n\in \text{span}(A)$. Since $\mu(x) \leq \rho(x)$ for $x\in X$, $\mu$ is a continuous seminorm on $(X, \tau)$. By Theorem 3.5 in \cite{flctopology}, $\mu$ is  continuous  on $(X, \tau_{F})$. If $\mu(z)<1$ for $z\in Z$, then there exists integers $c_{n_1}, c_{n_2}, ..., c_{n_N}$ such that $z=\sum_{k=1}^{N}c_{n_k}z_{n_k}$  and $\mu(z)=\sum_{k=1}^{N}|c_{n_k}|<1$. So $c_{n_k}=0$ for all $k$. Consequently, $z=0$. Therefore $\{0\}= \mu^{-1}[0, 1)\cap Z$ is open in $(Z, \tau_{F}|_Z)$. Hence $Z$  is a discrete subgroup of $(X, \tau_{F})$.\end{proof}

\begin{theorem}\label{metrizability of weak topology in elcs}Let $(X, \tau)$ be an elcs with the flc topology $\tau_{F}$ and weak topology $\tau_w$. If $\tau_{w}$ is metrizable, then $\tau_{F}=\tau_w$. Converse holds for countable elcs.  \end{theorem}
\begin{proof}
	Suppose $\tau_w$ is metrizable. Then $(X, \tau_w)$ is a Mackey space (3.4 in \cite{tvsschaefer}, p. 132). Since $(X, \tau_{F})^*=(X, \tau_w)^*$, we have $\tau_w=\tau_{F}$ (Theorem 4.35 in \cite{faaidg}, p. 120). 
	
	Conversely, suppose $(X, \tau)$ is a countable elcs and $\tau_{F}=\tau_{w}$. By Proposition \ref{coinicidence of finest and weak topology}, the dimension of the quotient space $X/X_{fin}^\rho$ is finite for every continuous extended seminorm $\rho$. So by Theorem 5.10 in \cite{flctopology}, $\tau_{F}$ is metrizable. Consequently, $\tau_{w}$ is metrizable. \end{proof} 

Suppose $(X, \parallel\cdot\parallel)$ is an enls. Then the weak topology for the normed space $(X_{fin}, \parallel\cdot\parallel)$ coincides with the subspace topology $\tau_w|_{X_{fin}}$, where $\tau_w$ is the weak topology for $(X, \parallel\cdot\parallel)$. So if $\tau_w= \tau_F$ on $X$, then $\tau_{w}|_{X_{fin}}=\tau_F|_{X_{fin}}= \tau_{ \parallel\cdot\parallel}$ on $X_{fin}$. Therefore the dimension of finite subspace is finite as weak topology for the finite subspace is normable. The next theorem now follows from this discussion and all the above results of this section.
\begin{theorem}\label{metrizability of weak topology in enls} Let $(X,\parallel\cdot\parallel)$ be an enls. Suppose $\tau_{F}$ and $\tau_w$ are the flc topology and weak topology for $X$, respectively. Then the following statements are equivalent.
	\begin{itemize}
		\item[(1)] $\tau_{F}=\tau_w$;
		\item[(2)] $\tau_{w}$ is metrizable;
		\item[(3)] $\tau_w$ is normable;
		\item[(4)] $X$ is of finite dimensional.\end{itemize}\end{theorem}
	
It is  known that the weak$^*$ topology for a locally convex space $X$ is metrizable (normable) if and only if the dimension of $X$ is countable (finite). So by Remark \ref{finite space have same weak  and weak$^*$ topology}, we have the following result.
\begin{theorem}\label{normability of weak^* topology}  
		Let $(X, \tau)$ be an elcs with the dual $X^*$ and weak$^*$ topology $\tau_{w^*}$. Then $\tau_{w^*}$ is metrizable (normable) if and only if the dimension of $X$ is countable (finite).\end{theorem}

Suppose $(X, \parallel\cdot\parallel)$ and $(Y, \parallel\cdot\parallel)$ are two extended normed linear spaces. Then for a continuous linear operator $T:X\to Y$ we define, $$\parallel T\parallel_{op}= \sup\{\parallel Tx\parallel: \parallel x\parallel\leq 1\}.$$
Note that if $T_1$ and $T_2$ are continuous linear operators on an enls $X$ such that $T_1|_{X_{fin}}=T_2|_{X_{fin}}$, then $\parallel T_1\parallel_{op}=\parallel T_2\parallel_{op}$. So $\parallel\cdot\parallel_{op}$ is a seminorm on $X^*$ (it become a norm if and only if $X=X_{fin}$). However, following Beer \cite{nwiv}, we call $\parallel T\parallel_{op}$ to be the operator norm of $T$. 

Recall that a normed linear space $X$ is separable if and only if the weak$^*$ topology on the closed unit ball of $X^*$ is metrizable. However, if an enls $X$ is separable, then $X$ must be a normed linear space (Proposition 3.10 in \cite{nwiv}). The next theorem relate the separability of the finest space $(X, \tau_{F})$ of an enls $X$ to the metrizability of the weak$^*$ topology on the closed unit ball $B_{X^*} = \{f\in X^*: \parallel f\parallel_{op}\leq 1\}$.

\begin{lemma}\label{separability of M}
	Let $(X, \parallel\cdot\parallel)$ be an enls with $X=X_{fin}\oplus M$ and flc topology $\tau_{F}$. Then $(M, \tau_{F}|_M)$ is separable if and only if the dimension of $M$ is countable.\end{lemma}
\begin{proof}Suppose $M=\text{span}\{x_n: n\in\mathbb{N}\}$. Then the collection of all linear combinations of $\{x_n: n\in\mathbb{N}\}$ with complex coefficient whose real and imaginary part are rational is dense in $(M, \tau_{F}|_M)$.
	
Conversely, suppose the dimension of $M$ is uncountable. Let $A=\{x_n:n\in\mathbb{N}\}$ be any countable subset of $M$. Then there exists an $x_0$ in $M$ such that $x_0$ is linearly independent of $A$. Suppose $Z$ is subspace of $X$ such that $X=Z\oplus \text{span}(x_0)$. Define a seminorm $\rho:X\to [0, \infty)$ by $\rho(\alpha x_0+y)=|\alpha|$ for $y\in Z$. Since $\rho\equiv0$ on $X_{fin}$, $\rho$ is continuous on $(X, \parallel\cdot\parallel)$. By Theorem  3.5 in \cite{flctopology}, $\rho$ is continuous on $(X, \tau_{F})$. Note that $\rho^{-1}((0,2))\cap M\in \tau_{F}|_M$ and $x_0\in \rho^{-1}((0, 2))$. As $x_n\notin\rho^{-1}((0,2))\cap M$ for any $n\in \mathbb{N}$, $A$ is not dense in $M$.\end{proof}

\begin{theorem}\label{metrizability of B_{X^*}}
 	Suppose $(X, \parallel\cdot\parallel)$ is an enls with flc topology $\tau_{F}$. Then the following statements are equivalent.
 	 \begin{itemize}
 	 	\item[(1)] $(B_{X^*}, \tau_{w^*})$ is metrizable;
 	 	\item[(2)] the dimension of $M$ is countable and $(X_{fin}, \parallel\cdot\parallel)$ is separable;
 	 	\item[(3)] both $(M, \tau_F|_M)$  and $(X_{fin}, \parallel\cdot\parallel)$ are separable;
 	 	\item[(4)] $(X, \tau_{F})$ is separable.
 	 \end{itemize}\end{theorem}
 \begin{proof}
 	(1)$\Rightarrow$(2). First we show that $(X_{fin}, \parallel\cdot\parallel)$ is separable. Let $\tau_{w^*}^f$ be the weak$^*$ topology on the closed unit ball $B_{X_{fin}^*}$ of $(X_{fin}, \parallel\cdot\parallel)^*$. We show that $(B_{X_{fin}^*}, \tau_{w^*}^f)$ is homeomorphic to a subspace of $(B_{X^*}, \tau_{w^*})$. Define  $\Psi:(B_{X_{fin}^*}, \tau_{w^*}^f)\rightarrow (B_{X^*}, \tau_{w^*})$ by $\Psi(f)=f'$, where \[ f'(x)=
 	\begin{cases}
 	\text{$f(x),$} &\quad\text{if $x\in X_{fin} $}\\
 	\text{0,} &\quad\text{if $x\in M.$ }\\
 	\end{cases}\]  Since $\parallel f\parallel_{op}=\parallel f'\parallel_{op}$, the map $\Psi$ is well defined. It is easy to see that $\Psi$ is an injective continuous map. By classical Banach Alaoglu Theorem (Theorem 3.21 in \cite{faaidg}, p. 71), $(B_{X_{fin}^*}, \tau_{w^*}^f)$ is compact. So $(B_{X_{fin}^*}, \tau_{w^*}^f)$ is homeomorphic to the subset $Z=\{f\in B_{X^*}: f|_M=0 \}$ of $B_{X^*}$. Hence $(B_{X_{fin}^*}, \tau_{w^*}^f)$ is metrizable. Thus $X_{fin}$ is separable.
 	
 	Now suppose that the dimension of $M$ is uncountable. Since $(B_{X^*}, \tau_{w^*})$ is metrizable, there exists a countable neighborhood base $\{U_n: n\in\mathbb{N}\}$ of $0$ in $(B_{X^*}, \tau_{w^*})$. For each $n\in\mathbb{N}$, there exists a finite subset $F_n$ of $X$ such that $\left((F_n)^\circ \cap B_{X^*}\right) \subseteq U_n$.  Let $x_0\in M$ be linearly independent of $\cup_{n\in\mathbb{N}}F_n$. Then one can find $f\in X^*$ such that $f\in B_{X^*}$ with $f(x)=0$ for every $x\in\cup_{n\in\mathbb{N}}F_n$ and $f(x_0)=2$. Which is not possible as $\{x_0\}^\circ\cap B_{X^*}$ is a neighborhood of $0$ in $(B_{X^*}, \tau_{w^*})$ but $f\in\left((F_n)^\circ) \cap B_{X^*}\right)\setminus\left(\{x_0\}^\circ\cap B_{X^*}\right)$ for all $n \in \mathbb{N}$.
 	
   The implications (2)$\Leftrightarrow$(3) follow from Lemma \ref{separability of M}.
   
   (3)$\Leftrightarrow$(4). Consider the map $\Psi:(X, \tau_{F})\rightarrow (X_{fin}, \parallel\cdot\parallel)\times (M, \tau_{F}|_M)$ by $\Psi(x)=(x_f, x_M),$ where $x=x_f+x_M$. Clearly, $\Psi$ is linear and bijective. If $U$ and $V$ are neighborhood of $0$ in $(X_{fin}, \parallel\cdot\parallel)$ and $(M, \tau_{F}|_M)$ respectively, then there exist $r>0$ and a continuous seminorm $\rho$ on $(M, \tau_{F}|_M)$ such that $B(0, r)=\{x\in X_{fin}: \parallel x\parallel<r \}\subseteq U$ and $\rho^{-1}([0,r))\subseteq V$. Define a seminorm $q:X\rightarrow[0, \infty)$ by $q(x)=\parallel x_f\parallel+ \rho(x_M)$. Since $q(x)\leq \parallel x\parallel$ for $x\in X$, we have $q$ is a continuous seminorm on $(X,\parallel\cdot\parallel)$. By Theorem 3.5 of \cite{flctopology}, $q$ is a continuous seminorm on $(X, \tau_{F})$. Also $\Psi(q^{-1}[0, r))\subseteq B(0, r)\times \rho^{-1}([0, r)) \subseteq U\times V$. So $\Psi$ is continuous.
 
 For the continuity of $\Psi^{-1}$, let $\rho$ be a continuous seminorm on $(X, \tau_{F})$ and $r>0$. Then $\rho$ is continuous on $(X, \parallel\cdot\parallel)$. Consequently, $(\rho^{-1}([0, \frac{r}{2}))\cap X_{fin})\times(\rho^{-1}([0, \frac{r}{2}))\cap M)$ is a neighborhood of $0$ in $(X_{fin}, \parallel\cdot\parallel)\times (M, \tau_{F}|_M)$ and  $(\rho^{-1}([0, \frac{r}{2}))\cap X_{fin})\times(\rho^{-1}([0, \frac{r}{2}))\cap M)\subseteq\Psi(\rho^{-1}[0, r))$. Therefore $\Psi^{-1}$ is continuous. Hence $(X, \tau_{F})$ is isomorphic to $(X_{fin}, \parallel\cdot\parallel)\times (M, \tau_{F}|_M)$. Which completes the proof.   

(2)$\Rightarrow$(1). Let dimension of $M$ be countable and let $(X_{fin}, \parallel\cdot\parallel)$ be separable. By Exercise 8.201(e) in \cite{tvsnarici}, p. 271 and Theorem \ref{normability of weak^* topology}, the product space $(B_{X_{fin}^*}, \tau_{w^*}^f)\times (M^*, \tau_{w^*})$ is metrizable. Consider  the map $\Psi:(B_{X^*}, \tau_{w^*})\to(B_{X_{fin}^*}, \tau_{w^*}^f)\times (M^*, \tau_{w^*})$ by $\Psi(f)=(f|_{X_{fin}}, f|_M)$. Then by Theorem 4.3 in \cite{nwiv}, $\Psi$ is  bijective. Since a net $f_\lambda\to f$ pointwise in $X^*$ if and only if both $f_\lambda|_{X_{fin}}\to f|_{X_{fin}}$ pointwise in $X_{fin}^*$ and $f_\lambda|_{M}\to f|_{M}$ pointwise in $M^*$, $\Psi$ is a homeomorphism. Hence $(B_{X^*}, \tau_{w^*})$ is metrizable. \end{proof}

We end this section with an application of our result to function spaces. Let $C(X)$ denote the set of all real-valued continuous functions on a metric space $(X,d)$. A family of non-empty subsets of $X$ is said to be a \textit{bornology} on $X$ if it covers $X$ and is closed under finite union and taking subsets of its members. A subfamily $\mathcal{B}_0$ of $\mathcal{B}$ is said be a \textit{closed base} for $\mathcal{B}$ if every element of $\mathcal{B}_0$ is closed and $\mathcal{B}_0$ is cofinal in $\mathcal{B}$ under the set inclusion.

The most commonly used topology on $C(X)$ is the classical topology of uniform convergence on $\mathcal{B}$, usually denoted by $\tau_{\mathcal{B}}$ (see \cite{Suc} for the definition). 

The topology $\tau_{\mathcal{B}}$ is actually induced by the collection $\left\lbrace \rho_B:B\in \mathcal{B}_0 \right\rbrace$ of extended seminorms on $C(X)$, where $\rho_B(f) =\sup_{x\in B}|f(x)|$ for each $f \in C(X)$. Consequently, $(C(X), \tau_{\mathcal{B}})$ is an elcs.

For  a bornology $\mathcal{B}$ on $(X,d)$ with a closed base, in \cite{Suc}, Beer and Levi introduced a variational form of $\tau_{\mathcal{B}}$ called  the topology of strong uniform convergence on $\mathcal{B}$, denoted by $\tau^{s}_{\mathcal{B}}$ (see \cite{Suc} for the definition). 

The topology $\tau_\mathcal{B}^s$ on $C(X)$ can also be seen to be induced by the extended seminorms of the form $$\rho_B^s(f) =\inf_{\delta>0}\left\lbrace \sup_{x\in B^\delta}|f(x)|\right\rbrace, $$ where $B\in \mathcal{B}_0$, and $B^{\delta} = \{x \in X: d(x,y) < \delta \text{ for some } y \in B\}$.    

In general, $\tau_{\mathcal{B}}$ and $\tau_{\mathcal{B}}^{s}$ are not equal on $C(X)$. However, if the bornology $\mathcal{B}$ is shielded from closed sets, then $\tau_\mathcal{B} =\tau_\mathcal{B}^s$ on $C(X)$ (see, Theorem 4.1 in \cite{Ucucas}).

 \begin{proposition}\label{codimension of finite subspace with respect to B in tauBs}
 	Suppose $(X,d)$ is a metric space and $B\subseteq X$ is closed. Then $B$ is compact if and only if the quotient space $C(X)/C(X)_{fin}^{B}$ is finite dimensional, where $C(X)_{fin}^{B}=\{f\in C(X): \sup_{x\in B}|f(x)|<\infty\}$.\end{proposition}
 \begin{proof}
 	If $B$ is compact, then $C(X)_{fin}^{B}=C(X)$. Conversely, suppose $B$ is not compact, then there exists a countable subset $A=\{x_n:n\in\mathbb{N}\}$ of $B$ which is closed and discrete in $(X,d)$. For every $m\in\mathbb{N}$, define a function $g_m$ on $A$ by  
 	\[g_m(x_n)=
 	\begin{cases}
 	\text{$0$;} &\quad\text{ if $n<m$}\\
 	\text{$n^{2m}$;} &\quad\text{if $n\geq m$.}\\
 	\end{cases}
 	\]
 	Clearly, $g_m$ is continuous on $A$ for every $m\in \mathbb{N}$. By Tietze extension theorem, there exists $f_m\in C(X)$ such that $f_m|_A=g_m$ for each $m\in\mathbb{N}$. It is easy to see that $L=\{f_m:m\in\mathbb{N}\}$ is a linearly independent subset of $C(X)$. Let $h=\sum_{j=1}^{k} c_jf_j$ for scalars $c_1, c_2,...,c_k$. We can assume that $c_k>0$ and $k\geq 2$. Suppose $M>0$. Then there exists an $n\in \mathbb{N}$ such that $n>\max\{\frac{|c_j|}{c_k}, \frac{1}{c_k}, k^2, M : 1\leq j\leq k\}$. Note that $$\frac{n}{c_k}+\sum_{j=1}^{k-1}\frac{|c_jn^{2j}|}{c_k}\leq n^2+\sum_{j=1}^{k-1}n.n^{2(k-1)}\leq n^{2k-1}+n^{2k-1}\left( \frac{k^2}{2}\right)\leq n^{2k}.$$
 	Therefore $h(x_n)=\sum_{j=1}^{k}c_jn^{2j}\geq c_k n^{2k}-\sum_{j=1}^{k-1}|c_j|n^{2j} \geq n >M$. Hence $h\notin C(X)_{fin}^{B}$. Consequently, the dimension of the quotient space $C(X)/C(X)_{fin}^{B}$ is not finite. We arrive at a contradiction. \end{proof}
  
\begin{theorem}
	Suppose $\mathcal{B}$ is a bornology on a metric space $(X,d)$ with a closed base. Then the following statements are equivalent.
\begin{itemize}
	\item[(1)] The flc topology for $(C(X), \tau_{\mathcal{B}}^s)$ is metrizable;
	\item[(2)] the flc topology for $(C(X), \tau_{\mathcal{B}})$ is metrizable;
	\item[(3)] $\mathcal{B}$ has a countable compact base, that is, it has a countable base consisting of compact sets. 
\end{itemize}\end{theorem} 
\begin{proof}
 (1)$\Rightarrow$ (3). Suppose $\tau_F$ is the flc topology for $(C(X), \tau_{\mathcal{B}}^s)$ and $B\in \mathcal{B}$ is closed. Since $\rho_B$ is a continuous extended seminorm on $(C(X), \tau_{\mathcal{B}}^s)$ and $\tau_F$ is metrizable, the dimension of $C(X)/C(X)_{fin}^{B}$ is finite (see, Theorem 5.10, (a)$\implies$(b) in \cite{flctopology}). By Proposition \ref{codimension of finite subspace with respect to B in tauBs}, $B$ is compact. Therefore $\mathcal{B}$ has a compact base. Consequently, $\rho_{B}^s$ is a seminorm on $C(X)$ for every $B\in\mathcal{B}$. Hence $\tau_{\mathcal{B}}^s=\tau_F$ is metrizable. Thus by Theorem 7.1 in \cite{Suc}, $\mathcal{B}$ has a countable base.  
 
 (3)$\Rightarrow$ (1). If $\mathcal{B}$ has a compact base, then $(C(X), \tau_{\mathcal{B}}^s)$ is a locally convex space. Thus the result follows from Theorem 7.1 in \cite{Suc}.
 
The equivalence (2)$\Leftrightarrow$ (3) can be proved in a manner similar to (1)$\Leftrightarrow$ (3) by using Theorem 4.4.2 in \cite{MN}.  
\end{proof}

 \begin{definition}\normalfont(\cite{tvsnarici})  Suppose $(X, \tau)$ is a locally convex space and $A\subseteq X$. Then $A$ is said to be \textit{totally bounded} in $X$ if for every  neighborhood $U$ of $0$, there exists a finite set $F\subseteq X$ such that $A\subseteq F+U$. 
 \end{definition}
 
 \begin{theorem}\label{coincidence of weak and flc topology in tauBs}
 	Suppose $\mathcal{B}$ is a bornology on a metric space $(X,d)$ with a closed base. Then the flc topology $\tau_{F}$ for $(C(X), \tau_{\mathcal{B}}^s)$ coincides with its weak topology if and only if $\mathcal{B}=\mathcal{F}$,  where $\mathcal{F}$ is the bornology of all finite subsets of $X$.   
 \end{theorem}
 \begin{proof} Suppose the flc topology $\tau_{F}$ for $(C(X), \tau_{\mathcal{B}}^s)$ coincides with its weak topology. If $B\in\mathcal{B}$ is closed, then $\rho_B$ is a continuous extended seminorm on $(C(X), \tau_\mathcal{B}^s)$. By Proposition \ref{coinicidence of finest and weak topology}, the dimension of  $C(X)/C(X)_{fin}^{B}$ is finite as $C(X)_{fin}^{B}=\{f\in C(X): \rho_{B}(f)<\infty\}$. So by Proposition \ref{coincidence of weak and flc topology in tauBs}, $B$ is compact. Hence $\rho_B$ is a continuous seminorm on $(C(X), \tau_{\mathcal{B}}^s)$. By Theorem 3.5 in \cite{flctopology}, $\rho_B$ is a continuous seminorm on $(C(X), \tau_F)$. Let $A=\{f\in C(X):\sup_{x\in X}|f(x)|\leq 1\}$. Then $A$ is bounded in the elcs $(C(X), \tau_{\mathcal{B}}^s)$ (see, Definition \ref{bounded set in an elcs}). Consequently, $A$ is bounded in $(C(X), \tau_F)$. By Theorem 8.2.8 in \cite{tvsnarici}, p. 231 , $A$ is totally bounded in $(C(X), \tau_F)$. Therefore there exists a finite set $F=\{f_1, ..., f_m\}\subseteq C(X)$ such that $A\subseteq F+ \rho_{B}^{-1}([0, \frac{1}{2}])$. If $B$ contains more than $m$ elements, say $B\supseteq \{x_1,..., x_m\}$, then there exists $f\in A$ such that for every $1\leq i\leq m$,  
 	\[f(x_i)=
 	\begin{cases}
 	\text{$1$;} &\quad\text{ if $f_i(x_i)<0$}\\
 	\text{$-1$;} &\quad\text{if $f_i(x_i)\geq0$.}\\
 	\end{cases}
 	\]    	
 	Since $f\in A$, $f=f_j+h$ for some $1\leq j\leq m$ and $h\in \rho_{B}^{-1}([0, \frac{1}{2}])$. Which is not possible as $|h(x_j)|=|f(x_j)-f_j(x_j)|>\frac{1}{2}$.  
 	
 	Conversely, if $\mathcal{B}=\mathcal{F}$, then by Theorem 4.1 in \cite{Ucucas}, $\tau_\mathcal{B}^s=\tau_{\mathcal{B}}$. Therefore $\tau_{\mathcal{B}}^s=\tau_p=\tau_F$ on $C(X)$, where $\tau_p$ is the topology of pointwise convergence. It is well known that $(C(X), \tau_p)$ carries its weak topology as for every $x\in X$, the map $x\mapsto f(x)$ for $f\in C(X)$ is a continuous linear functional on $(C(X), \tau_p)$. Which completes the proof.     
 \end{proof}
 
 \begin{theorem}\label{coincidence of weak and flc topology in tauB}
 	Suppose $\mathcal{B}$ is a bornology on a metric space $(X,d)$ with a closed base. Then the flc topology $\tau_{F}$ for $(C(X), \tau_{\mathcal{B}})$ coincides with its weak topology if and only if $\mathcal{B}=\mathcal{F}$,  where $\mathcal{F}$ is the bornology of all finite subsets of $X$.   
 \end{theorem}
 \begin{proof} It is similar to the proof of Theorem \ref{coincidence of weak and flc topology in tauBs}.
 \end{proof}
 
 \section{Compactness and equicontinuity}   In this section, we discuss the equicontinuity of subsets of the dual $X^*$ of an elcs $(X, \tau)$ and show that the extended topology $\tau$ is actually induced by the polar of equicontinuous subsets of $X^*$. We also investigate, the classical Banach-Alaoglu Theorem in the extended case. Finally, we show that James Compactness Criterion (Theorem 3.55 in \cite{faaidg}, p. 84), Eberlein-\v{S}mulian Theorem (Theorem  4.47 in \cite{faaidg}, p. 128) for weak compactness also hold for an extended Banach space.      
 
\begin{definition}\normalfont	Let $(X, \tau)$ be an elcs and $A\subseteq X^*$. Then $A$ is \textit{equicontinuous} on $(X, \tau)$ if there exists a neighborhood $U$ of $0$ such that $|f(x)|\leq 1$ for every $f\in A$ and $x\in U$, that is, $A\subseteq U^\circ$.\end{definition} 
 The following points are  easy to verify for  $A\subseteq X^*$ of an elcs $(X, \tau)$.
 \begin{itemize}
 	\item[(1)] $A$ is equicontinuous on $(X, \tau)$ if and only if $A_\circ$ is a neighborhood of $0$ in $(X, \tau)$.
 	\item[(2)] If $A\subseteq B\subseteq X^*$ and $B$ is equicontinuous, then $A$ is also equicontinuous.
 	\item[(3)] If $X$ is an enls, then  $A\subseteq X^*$ is equicontinuous if and only if there exists an $M>0$ such that $\parallel f\parallel_{op}\leq M$ for each $f\in A$. In particular, the closed unit ball $B_{X^*}=\{f\in X^*: \parallel f\parallel_{op}\leq 1\}$ is always equicontinuous on $X$.
 	\item[(4)] If $X$ is an enls and $x\in X\setminus X_{fin}$, then for any $n\in\mathbb{N}$ there exists a linear functional $f_n\in B_{X^*}$ such that $f_n(x)=n$. So $B_{X^*}$ is not pointwise bounded. Hence an equicontinuous family on an elcs may not be pointwise bounded.   
 \end{itemize}

 \begin{theorem}\label{equiconouty and pointwise bounded}	Let $(X, \tau)$ be an elcs with the flc topology $\tau_{F}$, and let $X=X_{fin}\oplus M$. Then the following statements are equivalent.  
 	\begin{itemize} 
 		\item[(1)] Every equicontinuous family $A\subseteq X^*$ on $(X, \tau)$ is equicontinuous on $(X, \tau_{F})$;  
 		\item[(2)] every equicontinuous family $A\subseteq X^*$ on $(X, \tau)$ is pointwise bounded;
 		\item[(3)] $0\in\text{Core}(U)$ for every neighborhood $U$ of $0$;
 		\item[(4)] $X=X_{fin}$, that is, $(X, \tau)$ is a locally convex space.	\end{itemize}\end{theorem}
 
 \begin{proof}
 	The implications (4)$\Rightarrow$(1)$\Rightarrow$(2) are obvious as every equicontinuous family on a locally convex space is pointwise bounded.
 	
 	(2)$\Rightarrow$(3). Let $U$ be an absolutely convex neighborhood of $0$ in $(X, \tau)$. Then $U\cap X_{fin}$ is a neighborhood of $0$ in the locally convex space $(X_{fin}, \tau|_{X_{fin}})$. So $U\cap X_{fin}$ is an absorbing set in $X_{fin}$. If $0\notin\text{Core}(U)$, then there exists an $x_0\in M$ such that $\alpha x_0\notin U$ for each non-zero scalar $\alpha$. Since $U$ is a neighborhood of $0$ in $(X, \tau)$, $U^\circ$ is equicontinuous  on $(X, \tau)$. By our assumptions, $U^\circ$ is pointwise bounded. So there exists an $n\in\mathbb{N}$ such that $U^\circ \subseteq n \{x_0\}^\circ$. Now, let $Z$ be a subspace of $X$ such that $X=Z\oplus \text{ span }\{x_0\}$. Since $x_0\in M$, there exists a  continuous extended seminorm $\rho$ on $(X, \tau)$ such that $\rho(x_0)=\infty$. So $X_{fin}^\rho\subseteq Z$. Then the linear functional defined by $f(x_z+\alpha x_0)=(n+1)\alpha$ for $x_z\in Z$ is in $X^*$ and $f\in U^\circ \setminus \left( n \{x_0\}^\circ\right)$. We arrive at a contradiction.
 	
 	(3)$\Rightarrow$(4). If $\rho$ is a continuous extended seminorm on $(X, \tau)$, then $\rho^{-1}([0, 1])$ is a neighborhood of $0$. By our assumption, $0\in\text{Core}\left(\rho^{-1}([0, 1])\right)$. Therefore $X_{fin}^\rho=X$. Hence $X_{fin}=X$. 
 \end{proof}
 
 From Theorem \ref{equiconouty and pointwise bounded}, it is clear that an equicontinuous family $A \text{ of } X^*$ need not be pointwise bounded. Therefore the map  $\rho_{A}(x)=\sup_{f\in A}|f(x)| \text{ for } x\in X$ may not be a seminorm on $X$ but it will always be an extended seminorm. 
 \begin{theorem}\label{equicontinuitous family generates extended topology} Let $(X, \tau)$ be an elcs with the flc topology $\tau_{F}$. Then $\tau$ is induced by the collection $\{\rho_A: A\subseteq X^* \text{ is equiconinuous on } (X, \tau)\}$ of extended seminorms.\end{theorem}
 
 \begin{proof}
 	Suppose $A\subseteq X^*$ is equicontinuous on $(X, \tau)$. Then there exists a neighborhood $U$ of $0$ in $(X, \tau)$ such that $A\subseteq U^\circ$. Since $U\subseteq (U^\circ)_\circ$ and $A_\circ\subseteq \rho_A^{-1}([0, 1])$, we have $U\subseteq \rho_A^{-1}([0, 1])$. Hence $\rho_A^{-1}([0, 1])$ is a neighborhood of $0$ in $(X, \tau)$. Consequently, $\rho_A$ is continuous on $(X, \tau)$.
 	
 	Conversely, suppose $\rho$ is a continuous extended seminorm on $(X, \tau)$ and $V= \rho^{-1}([0, 1])$. Then there exists a subspace $M$ of $X$ such that  $X=X_{fin}^\rho\oplus M$. Let $z\notin V$, and let $z=z_F+z_M$ for $z_F\in X_{fin}^\rho$ and $z_M\in M$. We show that $z\notin\text{ Cl}_{\tau_F} (V)$. Define a seminorm $\mu$ on $X$ by $\mu(x)=\rho(x_M)+\nu(x_M)$ for $x\in X$, where $x=x_F+x_M$ for $x_F\in X_{fin}^\rho$, $x_M\in M$, and $\nu$ is seminorm on $M$ such that $\nu(z_M)=2$ whenever $z_m\neq 0$. Since $\mu(x)\leq \rho(x)$ for $x\in X$, $\mu$ is continuous on $(X, \tau)$. By Theorem 3.5 in \cite{flctopology}, $\mu$ is continuous on $(X, \tau_F)$. So $z\notin Cl_{\tau_F}(V)$ as $\mu^{-1}([0, 1])$ is closed in $(X, \tau_{F})$ with $V\subseteq \mu^{-1}([0, 1])$ and $\mu(z)>1$. Hence $V$ is closed in $(X, \tau_F)$. So by Bipolar Theorem (Theorem 8.3.8 in \cite{tvsnarici}, p. 235), $(V^\circ)_\circ=V$. It is easy to see that if $\rho_{V^\circ}(x)=k$ for $k>0$ and $x\in X$, then $\frac{1}{k}x\in (V^\circ)_\circ=V$. Consequently, $\rho(x)\leq k$. Therefore  $\rho(x)\leq \rho_{V^\circ}(x)$ for $x\in X$. Hence $\tau$ is induced by the collection $\{\rho_A: A\subseteq X^* \text{ is equiconinuous on } (X, \tau)\}$.            \end{proof} 	
 \begin{theorem}\label{equicontinuitous family generates finest topology}
 	Let $(X, \tau)$ be an elcs with the flc topology $\tau_{F}$. Then $\tau_{F}$ is induced by the collection  $\mathcal{A}=\{\rho_A : A\subseteq X^* \text{ is equiconinuous on } (X, \tau) \text{ and } 0\in \text{Core}(A_\circ) \}$. \end{theorem}
 \begin{proof}
 	First, we show that each $\rho_A\in\mathcal{A}$ is a seminorm on $X$. If $A\subseteq X^*$ is equicontinuous on $(X, \tau)$ and $0\in \text{Core}(A_\circ)$, then for every $x\in X$, there exists a $t>0$ such that $tx\in A_\circ$. Which implies that $\rho_{A}(x)\leq \frac{1}{t}$. Consequently, $\rho_{A}$ is a seminorm on $X$.
 	
 	Suppose $\tau'$ is the topology induced by the collection $\mathcal{A}$. Then $(X, \tau')$ is a locally convex space. By Theorem \ref{equicontinuitous family generates extended topology}, $\tau'\subseteq \tau$. Therefore $\tau'\subseteq \tau_F$. 
 	
 	Conversely, let $V$ be an absolutely convex, absorbing and closed neighborhood of $0$ in $(X, \tau_F)$. Then $V^\circ$ is equicontinuous on $(X, \tau)$ as $\tau_F\subseteq \tau$. Since $V$ is absorbing and $V\subseteq \left(V^\circ\right)_\circ$, $0\in \text{Core}\left(\left(V^\circ\right)_\circ\right)$. Consequently, $\rho_{V^\circ}$ is a continuous seminorm on $(X, \tau')$. By Bipolar Theorem, $(V^\circ)_\circ=V$. Then $\rho_{V^\circ}^{-1}([0, 1))\subseteq \left(V^\circ\right)_\circ=V$. So $V$ is a neighborhood of $0$ in $(X, \tau')$. Hence $\tau_F\subseteq \tau'$. \end{proof}	
 
 We conclude this section by considering some classical theorems on compactness in weak and weak$^*$ topology for an elcs. In general, the polar of a neighborhood of $0$ in an elcs may not be weak$^*$ compact.
 \begin{example}
 	Let $X$ be a nonzero vector space with the discrete extended norm $\parallel\cdot\parallel_{0, \infty}$. Then $B_{X}=\{0\}$. So $(B_{X})^\circ=X^*$ cannot be weak$^*$ compact.
 \end{example}
 \begin{theorem}\label{Alaoglu theorem} {\normalfont(Banach-Alaoglu Theorem)}	Suppose $(X, \tau)$ is an elcs with the flc topology $\tau_{F}$. If $U$ is an absolutely convex neighborhood of $0$ in $(X, \tau)$, then the following statements are equivalent.
 	\begin{itemize}
 		\item[(1)] $U^\circ$ is weak$^*$ compact;
 		\item[(2)] $0\in\text{Core}(U)$;
 		\item[(3)] $U$ is a neighborhood of $0$ in $(X, \tau_F)$\end{itemize}\end{theorem}
 \begin{proof}
 	(3)$\Rightarrow$(1). Follows from the Classical Alaoglu Theorem (Theorem 8.4.1 in \cite{tvsnarici}).
 	
 	(1)$\Rightarrow$(2). Since $U$ is a neighborhood of $0$ in $(X, \tau)$, $U^\circ$ is equicontinuous on $(X, \tau)$. If $U^\circ$ is weak$^*$ compact, then it is pointwise bounded. So $0\in \text{Core}(U)$ (see proof of (2)$\Rightarrow$(3) in Theorem \ref{equiconouty and pointwise bounded}).
 	
 	(2)$\Rightarrow$(3). If $0\in\text{Core}(U)$, then the Minkowski's functional $\mu_{U}$ is a continuous seminorm on $(X, \tau)$. By Theorem 3.5 in \cite{flctopology}, $\mu_{U}$ is a continuous seminorm on $(X, \tau_F)$. Hence $U$ is a neighborhood of $0$ in $(X, \tau_F)$. \end{proof}
 
 \begin{corollary}\label{Banach Alaoglu Theorem in an enls}
 	Suppose $(X, \parallel\cdot\parallel)$ is an elcs. Then the closed unit ball $B_{X^*}$ in $X^*$ is weak$^*$ compact if and only if $X$ is a normed linear space, that is, $X=X_{fin}$.\end{corollary}
 \begin{proof} It follows from the fact $(B_X)^\circ=B_{X^*}$ and Theorem \ref{Alaoglu theorem}, where $B_X$ is the closed unit ball in $(X, \parallel\cdot\parallel)$. 	\end{proof}
 
 \begin{theorem}\label{eberlin theorem}{\normalfont(Eberlein-\v{S}mulian Theorem, and James Compactness Criterion )} Let $(X, \parallel\cdot\parallel)$ be an extended Banach space with $X=X_{fin}\oplus M$, and let $A\subseteq X$ be weakly closed. Consider the following statements
 	\begin{itemize}
 		\item[(1)] $A$ is weakly compact;
 		\item[(2)] $A$ is weakly sequentially compact (every sequence in $A$ has a weakly convergent subsequence);
 		\item[(3)] $A$ is weakly countably compact (every countable infinite subset of $A$ has a cluster point in its weak topology);
 		\item[(4)] for every $f\in X^*$, there is an $x_0\in A$ such that $|f(x_0)|=\sup_{x\in A}|f(x)|$.
 	\end{itemize} Then $(1)\Leftrightarrow(2)\Leftrightarrow(3)\Rightarrow(4)$. In particular, if $A$ is convex, then $(4)\Rightarrow(1)$.\end{theorem}
 
 \begin{proof} Suppose $A$ satisfy any one of the given statements. Then $\sup_{x\in A}|f(x)|$ is finite for every $f\in X^*$. We show that $A\subseteq X_{fin}\oplus\text{span}(F)$ for a finite set $F$ in $M$. Suppose $(x_n)$ is a sequence of linearly independent elements in $M$ such that $y_n+x_n\in A$ for some $y_n\in X_{fin}$. Then there exists a $g\in X^*$ such that $g(x_n)=n$ for $n\in \mathbb{N}$ and $g(X_{fin})=0$. So $\sup_{x\in A}|g(x)|$ is not finite. We arrive at a contradiction. 
 	
 Let $Y=X_{fin}\oplus Z$, where $Z=\text{span}(F)$. Then $Y_{fin}=X_{fin}$. By Proposition 3.11 in \cite{nwiv}, $(Y, \parallel\cdot\parallel)$ is an extended Banach space.	Consider the norm $\trinorm\cdot\trinorm$ on $Y$ defined by $\trinorm x \trinorm =  \parallel x_F \parallel + \mu(x_Z)$ for $x\in Y$, where $x=x_F+x_Z$ for $x_F\in X_{fin}$, $x_Z\in Z$, and $\mu$ is any norm on $Z$. Then $(Y, \trinorm\cdot\trinorm)$ is a Banach space as both $(X_{fin}, \parallel\cdot\parallel)$ and $(Z, \mu)$ are Banach spaces. Since  $\trinorm x\trinorm=\parallel x\parallel$ for  $x\in Y_{fin}$, $\trinorm\cdot\trinorm$ is a continuous norm on $(Y, \parallel\cdot\parallel)$. Therefore by Corollary 1 in \cite{spoens}, the flc topology $\tau_{F_Y}$ for $(Y, \parallel\cdot\parallel)$ is induced by the norm $\trinorm \cdot\trinorm$.

 It follows from Theorem 4.1 in \cite{flctopology} that  $\tau_{F_Y}$ is equal to $\tau_{F}|_Y$.  Note that the weak topology for $(Y, \parallel\cdot\parallel)$ is equal to the subspace topology of the weak topology for  $(X, \parallel\cdot\parallel)$. Since $(Y, \tau_{F_Y})$ and $(Y, \parallel\cdot\parallel)$ have the same weak topology, both the spaces $(Y, \parallel\cdot\parallel)$ and $(Y, \trinorm\cdot\trinorm)$ also have the same weak topology. Therefore $A$ is weakly closed in $(Y, \trinorm\cdot\trinorm)$.  
 	
 The equivalences $(1)\Leftrightarrow(2)\Leftrightarrow(3)$ now follow by applying Theorem 4.47 in \cite{faaidg}, p. 128 on the space $(Y, \trinorm\cdot\trinorm)$ constructed above. 
 	
 $(1)\Rightarrow (4)$ is obvious as $f(A)$ is compact in $\mathbb{K}$ for every $f\in X^*$.
 
 In particular, if $A$ is convex and for every $f\in X^*$ there is a $x_0\in A$ such that $|f(x_0)|=\sup_{x\in A}|f(x)|$, then by Theorem 3.55 in \cite{faaidg}, p. 84, $A$ is  weakly compact in $(Y, \trinorm\cdot\trinorm)$. Therefore $A$ is weakly compact in $(X, \parallel\cdot\parallel)$. 	\end{proof}

\section{The topology of uniform convergence on bounded sets}

Recall that the generalization of the operator norm topology in a locally convex space is the \textit{strong topology} $\tau_s$ which is the topology of uniform convergence on bounded subsets of that locally convex space. In the case of an enls $X$, the topology of uniform convergence on bounded subsets of $X$ has been studied  in \cite{spoens}. This topology is denoted by $\tau_{ucb}$. 

In this section, we define and study the topology $\tau_{ucb}$ of uniform convergence on bounded subsets of an elcs $(X, \tau)$. We also find conditions for the metrizability and normablility of $(X^*,\tau_{ucb})$. We start with the following definition. 

\begin{definition}{\normalfont(\cite{flctopology})}\label{bounded set in an elcs} \normalfont	Suppose $(X, \tau)$ is an elcs. Then $A\subseteq X$ is said to be \textit{bounded} in $(X, \tau)$ if for every neighborhood $U$ of $0$, there exist $r>0$ and a finite set $F\subseteq A$ such that $A\subseteq F+rU$. \end{definition}
The following points about bounded sets in an elcs $(X,\tau)$ are easy to verify.
\begin{itemize}
	\item[(a)] Every finite set in an elcs is bounded.
	\item[(b)] If $A\subseteq B$ and $B$ is bounded, then $A$ is bounded.
	\item[(c)] A subspace of $X$ is bounded if and only if it is the zero subspace.
	\item[(d)] If  $x_n\to 0$ in $(X,\tau)$, then $\{x_n:n\in \mathbb{N}\}$ is  bounded in  $(X, \tau)$. 
    \item[(e)] If $\tau_F$ is the flc topology for $(X, \tau)$, then every bounded set in $(X, \tau)$ is bounded in $(X, \tau_F)$. 
\end{itemize}

\begin{proposition}\label{both tau and tauF may not have same bounded sets}
Suppose	$\tau_F$ is the flc topology for an elcs $(X, \tau)$. Then both $(X, \tau)$ and $(X, \tau_F)$ have the same collection of bounded sets if and only if $X=X_{fin}$, that is, $(X, \tau)$ is a locally convex space.
\end{proposition}
\begin{proof}
	Suppose $x\in X\setminus X_{fin}$, that is, $\rho(x)=\infty$ for some continuous extended seminorm $\rho$ on  $(X, \tau)$. Then $A=\{\frac{x}{n} :n\in\mathbb{N} \}$ is bounded in $(X, \tau_F)$ as $|f\left( \frac{x}{n}\right)|\leq |f(x)|$ for every $f\in X^*$. But it is not bounded in $(X, \tau)$ as $A\nsubseteq F+r\rho^{-1}([0, 1])$ for any $r>0$ and finite subset $F$ of $A$.
\end{proof}

\begin{definition}\label{definition of topology of uniform topology}\normalfont
	Let $(X, \tau)$ be an elcs. Then the topology $\tau_{ucb}$ on $X^*$, of uniform convergence on bounded subsets of $(X, \tau)$ is induced by the collection $\mathcal{P}=\{\rho_B: B ~\text{is bounded subset of}~ (X, \tau)\}$ of seminorms on $X^*$, where $\rho_B(f)=\sup_{x\in B}|f(x)| ~\text{for}~ f\in X^*$.\end{definition}
     
\noindent Let $(X, \tau)$ be an elcs with the flc topology $\tau_F$. Then by the strong topology $\tau_s$ on $X^*$, we mean the topology of uniform convergence on bounded subsets of $(X, \tau_F)$. Clearly, the topology $\tau_{ucb}$ is coarser than $\tau_s$. The following points regarding $\tau_{ucb}$ are easy to verify. 
\begin{itemize}
	\item[(1)] $(X^*, \tau_{ucb})$ is a locally convex space and $\mathcal{B}=\{B^\circ: B ~\text{is bounded in}~ (X, \tau)\}$ is a neighborhood base of $0$ in $(X^*, \tau_{ucb})$.
	\item[(2)] The weak$^*$ topology $\tau_{w^*}$ is coarser than $\tau_{ucb}$.
	\item[(3)] The following fact about the topology $\tau_{ucb}$ follows from Theorem 4.11 in \cite{nwiv}. If $(X, \parallel\cdot\parallel)$ is an extended normed space with $X=X_{fin}\oplus M$, then $(X^*, \tau_{ucb})$ is isomorphic to $(X_{fin}^*, \parallel\cdot\parallel_{op})\times (M^*, \tau_{w^*})$.
\end{itemize}

 In \cite{spoens}, authors proved that if $(X, \parallel\cdot\parallel)$ is an enls, then $\tau_{ucb}$ is normable if and only if $X$ is almost conventional, that is, the dimension of $X/X_{fin}$ is finite. We study this type of results in an elcs and also refine this result for an enls. We also find conditions on an elcs $(X, \tau)$ under which the strong topology on $X^*$ is normable.
 	
\begin{theorem}\label{normability of uniform topology in elcs} Let $(X, \tau)$ be an elcs. \begin{itemize}
		\item[(1)] If $\tau_{ucb}$ is normable, then the dimension of the quotient space $X/X_{fin}^\rho$ is finite for every continuous extended seminorm $\rho$ on $(X, \tau)$.
		\item[(2)] If $\tau_{ucb}$ is metrizable, then the dimension of the quotient space $X/X_{fin}^\rho$ is countable for every continuous extended seminorm $\rho$ on $(X, \tau)$.\end{itemize}\end{theorem}
\begin{proof}
		(1). Suppose $\tau_{ucb}$ is normable. By Theorem 6.2.1 in \cite{tvsnarici}, p. 160, there exists a bounded set $B$ in $(X, \tau)$ such that $B^\circ$ is a bounded neighborhood of $0$ in $(X^*, \tau_{ucb})$. Suppose the dimension of $X/X_{fin}^\rho$ is infinite for some continuous extended seminorm $\rho$ on $(X, \tau)$. So there exists an infinite dimensional subspace $M$ of $X$ such that $X=X_{fin}^\rho\oplus M$. Since $B$ is bounded in $(X, \tau)$, there exists an $x_0$ in $M$ linearly independent to $B$ as $B\subseteq \rho^{-1}([0, r])+F$ for a finite set $F$ and $r>0$. Moreover  since $B^\circ$ is bounded in $(X^*, \tau_{ucb})$, there exists an $m\in \mathbb{N}$ with $B^\circ\subseteq m\{x_0\}^\circ$. Which is not possible as we can find a $f\in X^*$ with $f|_B=0$ and $f(x_0)= m+1$.
		
		(2). Let $\tau_{ucb}$ be metrizable and  let $\{B_n^\circ : n\in\mathbb{N}\}$ be a countable neighborhood base of $0$ in $(X^*, \tau_{ucb})$ such that $B_n$ is bounded in $(X, \tau)$ for every $n\in\mathbb{N}$. If the dimension of $X/X_{fin}^\rho$ is not countable for some continuous extended seminorm $\rho$ on $(X, \tau)$, then there exists a subspace $M$ of $X$ such that $X=X_{fin}^\rho\oplus M$ and the dimension of $M$ is uncountable. For each $n\in\mathbb{N}$, there exists a finite subset $F_n$ of $M$ such that $B_n\subseteq X_{fin}^\rho+ F_n$. Take $F=\cup_{n\in\mathbb{N}} F_n$. Then there exists an $x_0$ in $M$ linearly independent to $F$ as the dimension of $M$ is  uncountable. Since $\{x_0\}^\circ$ is a neighborhood of $0$ in $(X^*, \tau_{ucb})$, $B_{n_0}^\circ\subseteq \{x_0\}^\circ$ for some $n_0$. Which is not possible as we can find a $f\in X^*$ with $f|_{B_{n_0}}=0$ and $f(x_0)= 2$.  	\end{proof}
	
In general, converse of (1) and (2) are not true (for example, one can take a locally convex space whose strong topology is not metrizable). We also may not replace $X_{fin}^\rho$ by $X_{fin}$ in (1) and (2) of Theorem \ref{normability of uniform topology in elcs} (see Example \ref{example in which stong topology coincide with the uniform topology}).  
 
\begin{theorem}\label{normability of uniform topology in enls}
 Let $(X, \parallel\cdot\parallel)$ be an enls with the flc topology $\tau_{F}$. Then the following statements are equivalent.
 \begin{itemize}
 	\item[(1)] $\tau_{ucb}$ is normable;
 	\item[(2)] $\tau_{F}$ is normable;
 	\item[(3)] $\tau_{F}$ is metrizable;
 	\item[(4)] $X$ is almost conventional.\end{itemize}\end{theorem}
  \begin{proof}
  	(1)$\Leftrightarrow$(4). It is given in Theorem 3 in \cite{spoens}.
  	
  	 (3)$\Leftrightarrow$(4). It follows from  Corollary 5.11 in \cite{flctopology}. (2)$\Rightarrow$ (3) is obvious.
  	 
  	 (4)$\Rightarrow$(2).  Suppose $X$ is almost conventional and $X=X_{fin}\oplus M$,  where the dimension of $M$ is finite. Consider the norm $\trinorm\cdot\trinorm:X\to[0, \infty)$ defined by $$\trinorm x\trinorm = \parallel x_f \parallel+ \mu(x_M),$$ where $x=x_f+x_M$ for $x_f\in X_{fin}$ and $x_M\in M$, and $\mu(\cdot)$ is any norm on $M$. Then $\trinorm\cdot\trinorm$ is continuous on $(X, \parallel\cdot\parallel)$ as $\trinorm x\trinorm\leq \parallel x\parallel$ for every $x\in X$.  If $f\in (X, \parallel\cdot\parallel)^*$, then $f|_{X_{fin}}\in (X_{fin}, \parallel\cdot\parallel)^*$ and $f|_M\in(M,\mu(.))^*$ as the dimension of $M$ is finite. Therefore $$|f(x=x_f+x_M)|\leq |f(x_f)|+|f(x_M)|\leq \left( \parallel f|_{X_{fin}}\parallel\right)  \parallel x_f\parallel +\parallel f|_M\parallel_\mu\mu(x_M),$$
  	 where $\parallel\cdot\parallel_\mu$ is the operator norm with respect to the norm $\mu(\cdot)$. Consequently, $|f(x)|\leq M \trinorm x \trinorm$, where $M= \max\{\parallel f|_{X_{fin}}\parallel, \parallel f|_M\parallel_\mu\}$. Therefore $f\in (X, \trinorm\cdot\trinorm)^*$. Hence $(X, \trinorm\cdot\trinorm)^*=(X,\parallel\cdot\parallel)^*$. By Theorem 5.16 in \cite{flctopology}, $\tau_{F}$ is induced by $\trinorm\cdot\trinorm$.\end{proof}
    
 \begin{theorem}
 	Let $(X, \tau)$ be a countable elcs with the flc topology $\tau_{F}$. Then 
 	\begin{itemize}
 		\item[(1)] the topology $\tau_{s}$ on $X^*$  is normable if and only if $\tau_{F}$ is normable.
 		\item[(2)] If $\tau_{ucb}$ is normable, then $\tau_{F}$ is normable.\end{itemize} \end{theorem}
 	
 	\begin{proof} (1). It is well known that the strong topology on the dual of a normable locally convex space is normable. Conversely,  let $\tau_s$ be normable. Then there exists a bounded set $B$ in $(X, \tau_F)$ such that $B^\circ$  is a bounded neighborhood of $0$ in $(X^*, \tau_{s})$.  So $B^\circ$ is pointwise bounded. We can assume that $B$ is closed, absolutely convex in $(X, \tau_F)$ (otherwise consider Cl$_{\tau_F}(ab(B))$). By Bipolar Theorem (Theorem 8.3.8 in \cite{tvsnarici}, p. 235) and Theorem 8.8.3 in \cite{tvsnarici}, p. 251,  $((B^\circ)_\circ)$ is absorbing and $(B^\circ)_\circ= B$. 
 		
 	Suppose $\{U_n: n\in\mathbb{N}\}$ is a countable neighborhood base of $0$ in $(X, \tau)$ with $U_{n+1}\subseteq U_n \text{ for every } n\in\mathbb{N}$. We first show that $B$ is a neighborhood of $0$ in $(X, \tau)$. To show this, it is sufficient to show that $\frac{U_{n}}{n}\subseteq B$ for some $n\in\mathbb{N}$. Suppose there exist $x_n\in U_n\setminus{nB}$ for every $n\in\mathbb{N}$. Then $x_n\rightarrow 0$ in $(X, \tau)$. Consequently, $x_n\rightarrow 0$ in $(X, \tau_F)$. Therefore $\{x_n: n\in\mathbb{N}\}$ is bounded in $(X, \tau_F)$. Since $B^\circ$ is bounded in $(X^*, \tau_s)$,  $B^\circ\subseteq m (\{x_n: n\in\mathbb{N}\})^\circ$ for some $m\in\mathbb{N}$. Therefore $(m\{x_n: n\in\mathbb{N}\}^\circ)_\circ\subseteq (B^\circ)_\circ=B$. Which is not possible as $\frac{x_m}{m}\in(m\{x_n: n\in\mathbb{N}\}^\circ)_\circ\subseteq B$. Hence $\frac{U_{n}}{n}\subseteq B$ for some $n\in\mathbb{N}$. Consequently, the Minkowski functional $\mu_{B}$ for $B$ is a continuous seminorm on $(X, \tau)$. By Theorem 3.5 in \cite{flctopology}, $\mu_{B}$ is continuous on $(X, \tau_F)$. Therefore $B$ is a bounded neighborhood of $0$ in $(X, \tau_F)$. Hence $\tau_F$ is normable.
 		
 	(2). Let  $\tau_{ucb}$ be normable. Then there exists a bounded set  $B$ in $(X, \tau)$ such that $B^\circ$  is a bounded neighborhood of $0$ in $(X^*, \tau_{ucb})$. So $B^\circ$ is pointwise bounded. By Bipolar Theorem (Theorem 8.3.8 in \cite{tvsnarici}, p. 235) and Theorem 8.8.3 in \cite{tvsnarici}, p. 251,  $(B^\circ)_\circ$ is absorbing and $(B^\circ)_\circ=\text{Cl}_{\tau_F}\left(\text{ab}(B)\right)$. Let $B'= \text{Cl}_{\tau_F}\left(\text{ab}(B)\right)$. Since both absolutely convex hull and closure of a bounded set are bounded in a locally convex space, $B'$ is bounded in $(X, \tau_{F})$. Now, repeat all the steps of converse of part (1) to show that $B'$ is a neighborhood of $0$ in $(X, \tau_{F})$. Which completes the proof.\end{proof}
  
  We have two locally convex topologies $\tau_{ucb}$ and $\tau_s$ on the dual $X^*$ of an elcs $(X, \tau)$. Clearly, $\tau_{ucb}\subseteq \tau_s$. A natural question arises here when $\tau_{ucb}=\tau_{s}$?  We show that for a fundamental elcs, both $\tau_{ucb}$ and $\tau_s$ coincide. We also give conditions on an elcs under which $\tau_{ucb}=\tau_s$. However, the question that whether $\tau_{ucb}=\tau_s$ for every elcs remains open.
    
  \begin{theorem}\label{coincidence of strong topology on a fundamental elcs} Let $(X, \tau)$ be a fundamental elcs such that $X=X_{fin}\oplus M$. Then $\tau_{ucb}=\tau_s$ on $X^*$.\end{theorem}
  \begin{proof}
  	Clearly, $\tau_{ucb}\subseteq \tau_s$. To show the reverse inclusion consider the projection maps $P_{X_{fin}}: (X, \tau_F)\to (X_{fin}, \tau_F|_{X_{fin}})$ and  $P_M:(X, \tau_F)\to (M, \tau_F|_M)$. We show that both $P_{X_{fin}}$ and $P_M$ are continuous. Let $\rho$ be a continuous seminorm on $(X_{fin}, \tau_F|_{X_{fin}})$. It follows from Theorem 4.1 in \cite{flctopology} that $\tau_F|_{X_{fin}}=\tau|_{X_{fin}}$. Therefore $\rho$ is continuous on $(X_{fin}, \tau|_{X_{fin}})$. Since $(X, \tau)$ is fundamental elcs, $X_{fin}$ is an open subspace of $(X, \tau)$. So the seminorm $\mu$ on $X$ defined by $\mu(x=x_f+x_M)=\rho(x_f)$ for  $x_f\in X_{fin}$ and $x_M\in M$ is continuous on $(X, \tau)$. Consequently, by Theorem 3.5 in \cite{flctopology}, $\mu$ is continuous on $(X, \tau_F)$. Note that $\rho(P_{X_{fin}}(x=x_f+x_m))=\rho(x_f)=\mu(x)$ for $x\in X$. Therefore $P_{X_{fin}}$ is continuous. Similarly, if $\lambda$ is a continuous seminorm on $(M, \tau_F|_M)$, then $\nu(x=x_f+x_M)=\lambda(x_M)$ for  $x_f\in X_{fin}$ and $x_M\in M$ is a continuous seminorm on $(X, \tau_F)$ and $\lambda(P_{M}(x=x_f+x_m))=\lambda(x_m)=\nu(x)$ for $x\in X$. Therefore $P_M$ is continuous. 
  	
  	Let $f_\alpha\to0$ in $(X^*, \tau_{ucb})$ and let $B$ be a bounded set in $(X, \tau_F)$. Then $B_1= P_{X_{fin}}(B)$  is bounded in $(X_{fin}, \tau_F|_{X_{fin}})$ and $B_2= P_M(B)$ is bounded in $(M, \tau_{F}|_M)$ with $B\subseteq B_1+B_2$. So $B_1$ is bounded in $(X_{fin}, \tau|_{X_{fin}})$ as $\tau|_{X_{fin}}=\tau_F|_{X_{fin}}$. Consequently, $B_1$ is bounded in $(X, \tau)$.  If $\{x_n:n\in\mathbb{N}\}\subseteq B_2$ is linearly independent, then there exists a linear functional $f$ on $M$ such that $f(x_n)=n$ for each $n \in \mathbb{N}$. Since  $(M, \tau|_M)$ is discrete, $f\in (M, \tau|_M)^*$. By Theorem 4.1 in \cite{flctopology}, $f\in (M, \tau_F|_M)^*$. Which is not possible as $B_2$ is bounded in $(M, \tau_F|_M)$. Therefore there exists a finite dimensional subspace $Y$ of $(M, \tau_F|_M)$ such that $B_2$ is a bounded subset of $Y$. Note that $$\sup_{x\in B}|f_\alpha(x)|\leq \sup_{x\in B_1}|f_\alpha(x)|+\sup_{x\in B_2}|f_\alpha|_Y(x)|.$$ 
  	It is easy to see that  $\sup_{x\in B_1}|f_\alpha(x)|\to 0$ and $f_\alpha(x)\to 0$ for $x\in Y$ as $f_\alpha\to 0$ in $(X^*, \tau_{ucb})$. Since weak$^*$ topology and strong topology on the dual of a finite dimensional locally convex space coincide, $\sup_{x\in A}|f_\alpha|_Y(x)|\to0$ for every bounded subset $A$ of $(Y, \tau_F|_Y)$. Consequently, $\sup_{x\in B_2}|f_\alpha|_Y(x)|\to 0$. Therefore $\sup_{x\in B}|f_\alpha(x)|\to 0$. Hence $\tau_{ucb}=\tau_s$.  \end{proof}
  
  \begin{corollary}\label{coincidence of strong topology on an enls}
  	Let $(X, \parallel\cdot\parallel)$ be an enls. Then $\tau_{ucb}= \tau_{s}$. \end{corollary}

 We now give an example of a countable elcs $(X,\tau)$ which is not fundamental, and for which $\tau_{ucb}=\tau_s$. In this example, all the three locally convex topologies $\tau_{F}$, $\tau_{s}$, and $\tau_{ucb}$ are normable but $\tau$ is not induced by an extended norm.
 
 \begin{example}\label{example in which stong topology coincide with the uniform topology}
 	Let $X$ be the collection of all real sequences which converges to $0$. For $n\in\mathbb{N}$, define $\rho_n:X\rightarrow [0,\infty]$ by
 	\[\rho_n((x^m))=
 	\begin{cases}
 	\text{$\infty$,} &\quad\text{ if $x^m\neq 0$  for some $1\leq m\leq n$}\\
 	\text{$\displaystyle{\sup_{m\in\mathbb{N}}}|x^m|$,} &\quad\text{if $x^m= 0$ for every $1\leq m\leq n$.}\\
 	\end{cases}
 	\] Suppose $\tau$ is the topology on $X$ induced by the collection $\mathcal{P}=\{\rho_n:n\in\mathbb{N} \}$ of extended norms. Then $(X,\tau)$ is an elcs and $X_{fin}=\{0\}$. By Example 5.12 in \cite{flctopology}, the flc topology $\tau_F$ for $(X, \tau)$ coincides with the supremum norm topology $\tau_{ \parallel\cdot\parallel_{\infty}}$. Let $B= \{(x^m)\in X : x^m\in \{-1, 0, 1\} \text{ for every } m\in\mathbb{N} \}$. Then $B$ is bounded in $(X, \tau)$. We show that $\text{Cl}_{\tau_F}(\text{ab}(B))= B_{\parallel\cdot\parallel_{\infty}}$, where $B_{\parallel\cdot\parallel_{\infty}}$ is the closed unit ball in $(X, \parallel\cdot\parallel_{\infty})$. Clearly, $\text{Cl}_{\tau_F}(\text{ab}(B))\subseteq B_{\parallel\cdot\parallel_{\infty}}$.  For the reverse inclusion, let $z=(z^m)\in B_{\parallel\cdot\parallel_{\infty}}$. Then $|z^m|\leq 1$ for every $m\in \mathbb{N}$.  We show that $(z^1, z^2, ..., z^m, 0, 0,...)\in \text{ab}(B)$ for every $m\in\mathbb{N}$. Note that if $(x^m)\in B$,  then $$(|z^1|, x^2, x^3 ...)= |z^1|(1, x^2, x^3,...)+(1-|z^1|)(0, x^2, x^3,...).$$ So $(|z^1|, x^2, x^3 ...)\in \text{ab}(B)$. Since $B=-B$ and $\text{ab}(B)$ is absolutely convex,  $$\left\lbrace (\pm z^1, x^2, x^3, ...) ~|~ x^{k}\in \{-1, 0, 1\} \text{ for } k\geq 2\right\rbrace \subseteq\text{ab}(B).$$  Similarly, if  $$\left\lbrace (\pm z^1, \pm z^2, ..., \pm z^{n-1}, x^n, x^{n+1}, x^{n+2},  ...)~ |~ x^{k}\in \{-1, 0, 1\} ~\text{for}~ k\geq n\right\rbrace \subseteq\text{ab}(B)$$ for some $n\in\mathbb{N}$, then  $$\left\lbrace (\pm z^1, \pm z^2, ..., \pm z^{n}, x^{n+1}, x^{n+2},  ...)~ |~ x^{k}\in \{-1, 0, 1\} ~\text{for}~ k\geq n+1\right\rbrace \subseteq\text{ab}(B).$$ Therefore by induction, $$\left\lbrace (\pm z^1, \pm z^2, ..., \pm z^{m}, x^{m+1}, x^{m+2}, ...)~ |~ x^{k}\in \{-1, 0, 1\} \text{ for } k\geq m+1\right\rbrace \subseteq\text{ab}(B)$$  for every $m\in \mathbb{N}$. So $P_m= (z^1, z^2, ..., z^m, 0, 0, 0, ...)\in \text{ab}(B)$ for every $m\in\mathbb{N}$. Since $P_m\to z$ in $(X, \parallel\cdot\parallel_{\infty})$,  $z\in \text{Cl}_{\tau_F}(\text{ab}(B))$. Consequently, $ B_{\parallel\cdot\parallel_{\infty}}\subseteq\text{Cl}_{\tau_F}(\text{ab}(B))$.
 	
 	 Now, let $B_{\parallel\cdot\parallel_{op}}$ be the closed unit ball in $(X^*, \parallel\cdot\parallel_{op})$, where $\parallel\cdot\parallel_{op}$ is the operator norm on the dual of normed space $(X, \parallel\cdot\parallel_{\infty})$. Then $(B)^\circ=\left( \text{Cl}_{\tau_F}(\text{ab}(B)) \right)^\circ=(B_{\parallel\cdot\parallel_{\infty}})^\circ= B_{\parallel\cdot\parallel_{op}}$. So $B_{\parallel\cdot\parallel_{op}}$ is a neighborhood of $0$ in $(X^*, \tau_{ucb})$. Consequently,  $\tau_{\parallel\cdot\parallel_{op}}\subseteq \tau_{ucb}$. Hence $\tau_{ucb}=\tau_{\parallel\cdot\parallel_{op}}$.     \end{example}
 
 The technique we used in the Example \ref{example in which stong topology coincide with the uniform topology} can be generalized to the arbitrary elcs as given in next theorem. 

\begin{theorem}\label{coincidence of strong topology on an elcs}
	Let $(X, \tau)$ be an elcs. Then 
	\begin{itemize}
		\item[(1)] $\tau_{ucb}=\tau_s$ if and only if for every bounded set $B$ in $(X, \tau_F)$, there exists a bounded set $B'$ in $(X, \tau)$ such that $\text{Cl}_{\tau_F}(\text{ab}(B))=\text{Cl}_{\tau_F}(\text{ab}(B'))$.
		\item[(2)] If $(X^*, \tau_{ucb})$ is barreled, then $\tau_{ucb}=\tau_s$.\end{itemize} \end{theorem}
\begin{proof}
	(1). Let $\tau_{ucb}=\tau_s$ and let $B$ be a bounded set in $(X, \tau_F)$. Since $B^\circ$ is a neighborhood of $0$ in $(X^*, \tau_s)$, there exists a bounded set $B_1$ in $(X, \tau)$ such that $(B_1)^\circ\subseteq (B)^\circ$. Take $B'= B_1\cap B$. Then $B'$ is bounded in $(X, \tau)$ and $(B')^\circ=(B_1\cap B)^\circ=(B_1)^\circ\cup (B)^\circ= B^\circ$. By Bipolar Theorem, $\text{Cl}_{\tau_F}(\text{ab}(B'))=\text{Cl}_{\tau_F}(\text{ab}(B))$. Conversely, if $B$ is a bounded set in $(X, \tau_F)$, then $\text{Cl}_{\tau_F}(\text{ab}(B'))=\text{Cl}_{\tau_F}(\text{ab}(B))$ for some bounded set $B'$ in $(X, \tau)$. So $(B)^\circ=(B')^\circ$. Therefore $(B)^\circ$ is a neighborhood of $0$ in $(X^*, \tau_{ucb})$. Hence $\tau_{ucb}=\tau_s$. 
	
   (2). If $B$ is a bounded set in $(X, \tau_F)$, then $B^\circ$ is absolutely convex and absorbing subset of $X^*$. Since $B^\circ$ is weak$^*$ closed, $B^\circ$ is a barrel in $(X^*, \tau_{ucb})$. So  $B^\circ$ is a neighborhood of $0$ in $(X^*, \tau_{ucb})$. Hence $\tau_{ucb}=\tau_{s}$. \end{proof}

\bibliographystyle{plain}
\bibliography{reference_file}		

\def\cprime{$'$} \def\cprime{$'$} \def\cprime{$'$}
\begin{thebibliography}{10}

\bibitem{tsoervms}
G.~Beer.
\newblock The structure of extended real-valued metric spaces.
\newblock {\em Set-Valued and Variational Analysis}, 21(4):591--602, 2013.

\bibitem{nwiv}
G.~Beer.
\newblock Norms with infinite values.
\newblock {\em Journal of Convex Analysis}, 22(1):37--60, 2015.

\bibitem{Suc}
G.~Beer and S.~Levi.
\newblock Strong uniform continuity.
\newblock {\em Journal of Mathematical Analysis and Applications},
  350(2):568--589, 2009.

\bibitem{Ucucas}
G.~Beer and S.~Levi.
\newblock Uniform continuity, uniform convergence, and shields.
\newblock {\em Set-Valued and Variational Analysis}, 18(3-4):251--275, 2010.

\bibitem{socsiens}
G.~Beer and J.~Vanderwerff.
\newblock Separation of convex sets in extended normed spaces.
\newblock {\em Journal of the Australian Mathematical Society}, 99(2):145--165,
  2015.

\bibitem{spoens}
G.~Beer and J.~Vanderwerff.
\newblock Structural properties of extended normed spaces.
\newblock {\em Set-Valued and Variational Analysis}, 23(4):613--630, 2015.

\bibitem{faaidg}
M.~Fabian, P.~Habala, P.~H{\'a}jek, V.~M. Santaluc{\'\i}a, J.~Pelant, and
  V.~Zizler.
\newblock {\em Functional analysis and infinite-dimensional geometry}.
\newblock Springer, 2001.

\bibitem{flctopology}
A.~Kumar and V.~Jindal.
\newblock The finest locally convex topology of an extended locally convex
  space.
\newblock {\em Topology and its Applications},
  doi.org/10.1016/j.topol.2022.108396, 2023.

\bibitem{MN}
R.~A. McCoy and I.~Ntantu.
\newblock {\em Topological properties of spaces of continuous functions}.
\newblock Lecture notes in mathematics, vol. 1315, Springer-Verlag, Berlin,
  1988.

\bibitem{Atgcfwt}
S.~Morris.
\newblock A topological group characterization of those locally convex spaces
  having their weak topology.
\newblock {\em Mathematische Annalen}, 195(2):330--331, 1971.

\bibitem{tvsnarici}
L.~Narici and E.~Beckenstein.
\newblock {\em Topological vector spaces}.
\newblock Second Edition, CRC Press, 2011.

\bibitem{lcsosborne}
M.~S. Osborne.
\newblock {\em Locally convex spaces}.
\newblock Springer, 2014.

\bibitem{esaetvs}
D.~Salas and S.~Tapia-Garc{\'\i}a.
\newblock Extended seminorms and extended topological vector spaces.
\newblock {\em Topology and its Applications}, 210:317--354, 2016.

\bibitem{tvsschaefer}
H.~H. Schaefer and M.~P. Wolff.
\newblock {\em Topological vector spaces}.
\newblock Second Edition, Springer-Verlag, New York, 1999.

\bibitem{willard}
S.~Willard.
\newblock {\em General topology}.
\newblock Addison-Wesley Publishing Co., Reading, Mass.-London-Don Mills, Ont.,
  1970.

\end{thebibliography}

\end{document}